\documentclass[11pt]{article}
\usepackage{amsmath,amssymb,amsthm}
\usepackage[margin=2.2cm]{geometry}
\usepackage{titlesec,hyperref}
\usepackage{fancyhdr} 
\newtheorem{theo}{Theorem}[section]
\newtheorem{lemm}[theo]{Lemma}
\newtheorem{defi}[theo]{Definition}

\newtheorem{prop}[theo]{Proposition}
\newtheorem{rema}[theo]{Remark}
\numberwithin{equation}{section}

\allowdisplaybreaks
\begin{document}
\title{Global Gevrey regularity and analyticity of a two-component shallow water system with Higher-order inertia operators}
\author{
Huijun $\mbox{He}^1$\footnote{email: hehuijun@mail2.sysu.edu.cn} \quad and\quad
 Zhaoyang $\mbox{Yin}^{1,2}$\footnote{email: mcsyzy@mail.sysu.edu.cn}\\
 $^1\mbox{Department}$ of Mathematics,
Sun Yat-sen University,\\ Guangzhou, 510275, China\\
$^2\mbox{Faculty}$ of Information Technology,\\ Macau University of Science and Technology, Macau, China}
\date{}
\maketitle
\begin{abstract}
 In this paper, we mainly  consider the  Gevrey regularity and analyticity of the solution to  a generalized two-component shallow water wave system with  higher-order inertia operators, namely, $m=(1-\partial_x^2)^su$ with $s>1$.  Firstly, we obtain the Gevrey regularity and analyticity for a short time. Secondly, we show the continuity of
 the data-to-solution map. Finally, we prove the  global Gevrey regularity and analyticity  in time.

 \end{abstract}
\tableofcontents
\section{Introduction}
\par

In this paper we mainly consider the analyticity and  Gevrey regularity  of  the solution to the following
generalized two-component shallow water wave system with higher-order inertia operators\cite{Escher-new}:
\begin{equation}\label{1}
 \ \ \ \ \ \ \ \  \ \ \ \ \left\{\begin{array}{ll}
m_{t}+um_{x}+amu_x =\alpha u_x-\kappa\rho\rho_{x} ,&t > 0,\,x\in \mathbb{R},\\
 \rho_{t}+u\rho_x+(a-1)u_x\rho=0, &t > 0,\,x\in \mathbb{R},\\
 m(t,x)=(1-\partial_x^2)^su(t,x),&t \geq 0,\,x\in \mathbb{R},\\
u(0,x) = u_{0}(x),&x\in \mathbb{R}, \\
\rho(0,x) = \rho_{0}(x),&x\in \mathbb{R},
\end{array}\right.
\end{equation}
where $ s>1$, $a\neq1$ is a real parameter, $\alpha$ is a constant which represents the vorticity of underlying flow, and $\kappa>0$ is an arbitrary real parameter.  The
 system  (1.1) is the generalization of the two component  water wave system  with $s=1$, namely, $m=(1-\partial_x^2)u=u-u_{xx}$ (see \cite{E-H-K-L},\cite{G-H-Y} and \cite{He-Yin}).

 For $\alpha=0, \rho\equiv0$, the system (1.1) becomes a  family of one-component equations
\begin{align}\label{B}
\left\{\begin{array}{lll}
       m_t+um_x+au_xm=0, &t > 0,\,x\in \mathbb{R},\\
       m(t,x)=(1-\partial_x^2)^su(t,x),&t\geq 0,\,x\in \mathbb{R},\\
       u(t,x)|_{t=0}=u_0(x).&x\in \mathbb{R}.
       \end{array}\right.
\end{align}
When $s=1,$  the equation  (\ref{B}) is called the b-equation. The b-equation possess a
number of structural phenomena which are shared by solutions of the family
of equations \cite{E-Y2008,H1,H2}. Recently, some authors were devoted to the study of  the Cauchy problem for the b-equation.  The local well-posedness of the b-equation was obtained by Escher and Yin in \cite{E-Y2008} and Gui, Liu and Tian in \cite{G-L-T}, respectively, on the line and Zhang and Yin
in \cite{Z-Y} on the circle. It also has global solutions \cite{E-Y2008,G-L-T,Z-Y} and solutions which blow up in finite time \cite{E-Y2008,G-L-T,Z-Y}. The
uniqueness and existence of global weak solution to the b-equation provided the initial data satisfies certain sign conditions
were obtained in \cite{E-Y2008,Z-Y}. However, there are just two members of this family
which are integrable \cite{Iv1}: the Camassa-Holm \cite{C-D1,C-D2} equation, when
a =2, and the Degasperis-Procesi \cite{D-P} equation,when $a=3.$ The Cauchy problem and initial-boundary value problem for the
Camassa-Holm equation have been studied extensively \cite{C-Ep,
C-Em,Dan2, E-Y1, E-Y2, L-O, Rb, y}. It has been shown that this
equation is locally well-posed \cite{C-Ep, C-Em,Dan2, L-O, Rb} for
initial data $u_{0}\in H^{q}(\mathbb{R}),\,q>\frac{3}{2}$. More
interestingly, it has global strong solutions \cite{Constantin, C-Ep,C-Em}
and also finite time blow-up solutions \cite{Constantin,C-E, C-Ep,C-Em,
C-Ez,Dan2,L-O,Rb}. On the other hand, it has global weak solutions in
$ H^1(\mathbb{R}) $ \cite{B-C,B-C1,C-Ei,C-M,X-Z}. Finite propagation speed and persistence properties for the Camassa-Holm equation have been studied in \cite{Constantin1,H-M-P}. After the Degasperis-Procesi equation was derived, many papers
were devoted to its study, cf.
\cite{C-K1,E-K,E-L-Y1,E-L-Y2,L,L-Y1,L-S,Matsuno,Yin1,Yin2,Yin3,Yin4}.
When $s=k\geq2$, the equation (\ref{B}) becomes higher-order b-equation. In \cite{M-Z-Z}, Mu et. al  studied the  local
 well-posedness and  global solutions for (\ref{B}) (under a scaling transformation) with $k=2$ in Sobolev
 spaces.  In \cite{C-H-K1}, Coclite et.al considered the cases $a=2, k\geq2, m=(1-\partial_x^2+\partial_x^4-...+(-1)^{k}\partial_x^{2k})u$ --- the higher-order
 Camassa-Holm equations, which describe the
exponential curves of the manifold of smooth orientation-preserving diffeomorphisms of the unit
circle in the plane. They established the existence of the  unique global weak solutions.

For $a=2$ and $\alpha=0$,  the system (1.1) becomes the two-component  Camassa-Holm equation. Several types of 2-component Camassa-Holm equations have
been studied in \cite{Chen, C-I, E,E-K-L,E-L-Y,G-K-Y, G-Y,G-Y1,G-Y2, G-Y3,G-Y4}.  These works
established the local well-posedness \cite{C-I, E-L-Y,G-K-Y,G-Y}, derived
precise blow-up scenarios \cite{E-L-Y,G-K-Y}, and proved that there
exist strong solutions which blow up in finite time
\cite{C-I, E-L-Y, G-Y}.  It also
has global strong solutions \cite{C-I,G-Y}. Moreover, it has global
weak solutions  \cite{G-Y1,G-Y2, G-Y3,G-Y4,T-Y1,T-Y2}.

The system (1.1) with $s>1$ was
recently introduced by
 Escher and Lyons in \cite{Escher-new}. It is the generalization of the same model (1.1) with $s=1$ in
   \cite{E-H-K-L}. In  \cite{E-H-K-L}, the authors  proved the local well-posedness of (1.1) with $s=1$ by using a geometrical framework and they studied the blow-up  scenarios and global strong solutions of (1.1) in the periodic case. In \cite{G-H-Y}, Guan et al.  studied the local well-posedness of (1.1)  with $s=1$  on the line in supercritical Besov spaces,
   and  several  blow-up results and the persistence properties.  In \cite{He-Yin}, He and Yin studied the local well-posedness  of (1.1) with $s=1$ in the critical Besov spaces on the line and the existence of analytic solutions of the system.
  In \cite{Escher-new}, for $s>1$, by a geometric approach, the authors gave a blow-up criteria to ensure the geodesic completeness on the circle with $s>\frac32, a=2, \kappa \geq 0$ for the $C^\infty$ initial data. In \cite{He-Yin-new}, the
  authors proved the local well-posedness results in Besov spaces with certain regularity condition, and gave some
  global existence results.

    In this paper, we focus on the analyticity and Gevrey regularity of the system (1.1). Many researchers  studied the  analyticity  for solutions  to
    Camassa-Holm type systems, cf. \cite{BHP,HM} and \cite{Yan-Yin}. Recently, Luo and Yin \cite{Luo-Y} studied  the Gevrey regularity of solutions  to the Camssa-Holm  type  system  by  a generalized  Ovsyannikov theorem (see Lemma \ref{ACK-Luo} below).
    Applying this lemma, and following the ideas of \cite{Luo-Y}, we obtain the local analyticity and Gevrey regularity of the
    solutions to system (1.1). Also, we prove the continuity of the  data-to-solution map. Considering  the existence of the global strong solution of system (1.1), with the idea from Levermore and Oliver\cite{Lev-O} or Foias and Temam
    \cite{Foi-T}, we also study the global analyticity
    and  Gevrey regularity of this system.
    First, we will show that, its solution  is of analyticity
    and Gervey regularity at least for a short time. Then, we will also show that, under certain conditions, the solution will
    keep in  analyticity or Gevrey regularity for all $t\geq0$.

Our paper is organized as follows. In Section 2, we give some  preliminaries which will be used
 in Section 3. In Section 3, we establish the
local analyticity and Gevrey regularity of the Cauchy problem associated with
(1.1) and with (\ref{B}).
In  Section 4, we discuss global analyticity and Gevrey regularity.

\section{Preliminaries}
We consider the Cauchy problem for the above system which can be rewritten in the following abstract form:
\begin{align}\label{ACP}
  \left\{
  \begin{array}{lll}
  \frac{\mathrm d u}{\mathrm d t}=F(t,u(t)),\\
  u|_{t=0}=u_0.
  \end{array}\right.
   \end{align}
 \begin{lemm}\cite{BG,Ov1}
Let $\{X_\delta\}_{0<\delta<1}$ be a scale of decreasing Banach spaces, namely, for any
$\delta'<\delta$ we have $X_{\delta}\subset X_{\delta'}$ and  $\|\cdot\|_{\delta'} \leq\|\cdot\|_\delta,$
and let $T,R>0.$ For given $u_0\in X_1,$ assume that:
\begin{itemize}
 \item [(1)]
 If for $0<\delta'<\delta<1$ the function $t\mapsto u(t)$ is holomorphic in $|t|<T$ and continuous on $|t|\leq T$
 with values in $X_\delta$ and
 $$\sup\limits_{|t|<T} \|u(t)\|_{\delta}<R,$$
 then $t\mapsto F(t,u(t))$ is a holomorphic function on $|t|<T$ with values  in $X_{\delta'}.$
 \item [(2)]
  For any $0<\delta'<\delta<1$ and any $u,v\in \overline{B(u_0,R)} \subset X_\delta,$ there exists a positive constant $L$
  depending on $u_0$ and $R$ such that
  $$\sup\limits_{|t|<T} \|F(t,u)-F(t,v)\|_{\delta'}\leq \frac L{\delta-\delta'}\|u-v\|_{\delta}.$$
 \item [(3)]
 For any $0<\delta<1,$ there exists a positive constant $M$ depending  on $u_0$ such that
 $$\sup\limits_{|t|<T} \|F(t,u_0)\|_{\delta}\leq \frac M{1-\delta}.$$
 \end{itemize}
 Then there exist  a $T_0\in(0,T)$ and a unique solution to the Cauchy problem (\ref{ACP}), which for  every
 $\delta\in(0,1)$ is holomorphic  in $|t|<T_0(1-\delta)$ with values in $X_\delta.$
 \end{lemm}

This theorem was first proposed by Ovsyannikov in \cite{Ov1,Ov2,Ov3}. However, the original Ovsyannikov theorem becomes invalid for the Gervey class. Because this kind of spaces do not satisfy the condtion (2) of  the Ovsyannikov theorem. More precisely, for the Gevrey class, we see that
\begin{align}
 \sup\limits_{|t|<T} \|F(t,u)-F(t,v)\|_{\delta'}\leq \frac L{(\delta-\delta')^\sigma}\|u-v\|_{\delta},
    \end{align}
    with $\sigma\geq1.$ If $\sigma>1,$ the above inequality is weaker than the condition (2) because it is of nonlinear decay.
    In \cite{Luo-Y}, Luo and Yin established a new auxiliary function to obtain a generalized Ovsyannikov  theorem by modifying the proof of \cite{BHP}.

 \begin{lemm}\cite{Luo-Y}\label{ACK-Luo}
    Let $\{X_\delta\}_{0<\delta<1}$ be a scale of decreasing Banach spaces, namely, for any
$\delta'<\delta$ we have $X_{\delta}\subset X_{\delta'}$ and  $\|\cdot\|_{\delta'} \leq\|\cdot\|_\delta.$
Consider the Cauchy problem
\begin{align}\label{ACP1}
  \left\{
  \begin{array}{lll}
  \frac{\mathrm d u}{\mathrm d t}=F(t,u(t)),\\
  u|_{t=0}=u_0.
  \end{array}\right.
   \end{align}
Let $T,R>0,\sigma \geq1.$  For given $u_0\in X_1,$ assume that $F$ satisfies the following three conditions:
\begin{itemize}
 \item [(1)]
 If for $0<\delta'<\delta<1$ the function $t\mapsto u(t)$ is holomorphic in $|t|<T$ and continuous on $|t|<T$
 with values in $X_\delta$ and
 $$\sup\limits_{|t|<T} \|u(t)\|_{\delta}<R,$$
 then $t\mapsto F(t,u(t))$ is a holomorphic function on $|t|<T$ with values  in $X_{\delta'}.$
 \item [(2)]
  For any $0<\delta'<\delta<1$ and any $u,v\in \overline{B(u_0,R)} \subset X_\delta,$ there exists a positive constant $L$
  depending on $u_0$ and $R$ such that
  $$\sup\limits_{|t|<T} \|F(t,u)-F(t,v)\|_{\delta'}\leq \frac L{(\delta-\delta')^\sigma}\|u-v\|_{\delta}.$$
 \item [(3)]
 For any $0<\delta<1,$ there exists a positive constant $M$ depending  $u_0$ such that
 $$\sup\limits_{|t|<T} \|F(t,u_0)\|_{\delta}\leq \frac M{(1-\delta)^\sigma}.$$
 \end{itemize}
 Then there exists  a $T_0\in(0,T)$ and a unique solution $u(t)$ to the Cauchy problem (\ref{ACP1}), which for  every
 $\delta\in(0,1)$ is holomorphic  in $|t|<\frac{T_0(1-\delta)^\sigma}{2^\sigma-1}$ with values in $X_\delta.$
    \end{lemm}
\begin{rema}\label{T-estimate}
 In fact, $T_0=\min \{\frac1{2^{2\sigma+4}L},\frac{(2^\sigma-1)R}{(2^\sigma-1)2^{2\sigma+3}LR+M}\},$
which gives a lower bound of the lifespan.
\end{rema}
\begin{rema}
The upper-bound condition $``\delta<1"$ is not essential. Indeed, we can replace $1$ by any other positive $\delta_0$ and
obtain the similar result.
\end{rema}
\begin{rema}
   If $\sigma=1$ Lemma \ref{ACK-Luo} reduces to the so called abstract Cauchy-Kovalevsky theorem.
   \end{rema}

To apply Lemma \ref{ACK-Luo}, we first introduce the Sobolev-Gevrey spaces and recall some basic properties.
\begin{defi}\cite{Foi-T}\label{G-defi}
 Let $s$ be a real number and $\sigma,\delta>0.$ A function $f\in G^\delta_{\sigma,s}(\mathbb R^d)$ if and
only if  $f\in C^\infty(\mathbb R^d)$ and satisfies
$$\|f\|_{G^\delta_{\sigma,s}(\mathbb R^d)}=\Big(\int_{\mathbb R^d}(1+|\xi|^2)^s e^{2\delta(1+|\xi|^2)^{\frac1{2\sigma}}}|\hat{f}(\xi)|^2\mathrm d \xi\Big)^\frac12<\infty.$$
A function $f\in \bar G^\delta_{\sigma,s}(\mathbb R^d)$ if and
only if  $f\in C^\infty(\mathbb R^d)$ and satisfies
$$\|f\|_{\bar G^\delta_{\sigma,s}(\mathbb R^d)}=\Big(\int_{\mathbb R^d}(1+|\xi|^2)^s e^{2\delta |\xi|^{\frac1{\sigma}}}|\hat{f}(\xi)|^2\mathrm d \xi\Big)^\frac12<\infty.$$
   \end{defi}
\begin{defi}\cite{L-O}
A function is of  Gevrey class $\sigma\geq1,$ if there exist $\delta>0, r\geq0$ such that $f\in G^\delta_{\sigma,r}(\mathbb R^d).$ Denote the the functions of Gevrey class $\sigma$
$$f\in G_\sigma(\mathbb R^d):=\bigcup_{\delta>0, r\in \mathbb R} G^\delta_{\sigma,r}(\mathbb R^d).$$
\end{defi}
\begin{rema}
 Denoting the Fourier multiplier $ \mathrm e^{\delta A^{\frac1{\sigma}}}$  and  $ \mathrm e^{\delta (-\delta)^{\frac1{2\sigma}}}$ by
  $$e^{\delta A^{\frac1{\sigma}}}f= \mathcal{F}^{-1}(\mathrm e^{\delta(1+|\xi|^2)^{\frac1{2\sigma}}}\hat f)
   \ \ \mathrm{and}\ \ \
e^{\delta (-\delta)^{\frac1{2\sigma}}}f= \mathcal{F}^{-1}(\mathrm e^{\delta|\xi|^{\frac1{\sigma}}}\hat f)$$
  respectively,
  we deduce that
   $\|f\|_{G^\delta_{\sigma,s}(\mathbb R^d)}=\|\mathrm e^{\delta A^{\frac1{\sigma}}}f\|_{H^s(\mathbb R^d)}$ and
    $\|f\|_{\bar G^\delta_{\sigma,s}(\mathbb R^d)}=\|\mathrm e^{\delta (-\delta)^{\frac1{2\sigma}}}f\|_{H^s(\mathbb R^d)}.$
Note that  for $\sigma\geq1,$
$$|\xi|^\frac1\sigma\leq (1+|\xi|^2)^\frac1{2\sigma}\leq 1+|\xi|^\frac1\sigma.$$
It follows that
$$\|f\|_{\bar G^\delta_{\sigma,s}}\leq \|f\|_{G^\delta_{\sigma,s}}\leq \mathrm e^\delta \|f\|_{\bar G^\delta_{\sigma,s}}.$$
   For  $0<\sigma<1$,  it is called ultra-analytic function. If $\sigma=1,$ it is usual analytic function and $\delta$ is called the radius of analyticity. If $\sigma>1,$ it is  the Gevrey class function.
   \end{rema}
 \begin{prop}\label{embedding sequence}
   Let $0<\delta'<\delta,\,0<\sigma'<\sigma$ and $s'<s.$ From Definition \ref{G-defi}, one can check that
   $G^\delta_{\sigma, s}(\mathbb R^d)\hookrightarrow G^{\delta'}_{\sigma, s}(\mathbb R^d),
   G^\delta_{\sigma', s}(\mathbb R^d)\hookrightarrow G^\delta_{\sigma, s}(\mathbb R^d)$ and
   $G^\delta_{\sigma, s}(\mathbb R^d)\hookrightarrow G^\delta_{\sigma, s'}(\mathbb R^d),$ with the embedding constants $C=1.$
    \end{prop}

    Along the similar computations of the proof of Proposition 2.4--2.5 in   \cite{Luo-Y}, we can obtain the following
    two propositions.
 \begin{prop}\label{gradient estimate}
   Let $s$ be a real number and $\sigma>0.$ Assume that $0<\delta'<\delta.$ Then we have
   $$\|\partial_xf\|_{G^{\delta'}_{\sigma,s}(\mathbb R)}
   \leq \frac{e^{-\sigma}\sigma^\sigma}{(\delta-\delta')^\sigma}
   \|f\|_{G^\delta_{\sigma,s}(\mathbb R)}.$$
    \end{prop}
 \begin{prop}\label{productlaws}
 Let $s>\frac12, \sigma\geq1$ and $\delta>0.$
 Then $G^\delta_{\sigma,s}(\mathbb R)$ is an algebra. Moreover, there exists a constant $C_s$ such that
 $$\|fg\|_{G^\delta_{\sigma,s}(\mathbb R)}
 \leq C_s\|f\|_{G^\delta_{\sigma,s}(\mathbb R)}\|g\|_{G^\delta_{\sigma,s}(\mathbb R)}.$$
    \end{prop}

 Now we can state some known results of the system (\ref{1}), which will be used in sequel.
 \begin{lemm}\cite{He-Yin-new}\label{global-Sobolev1}
  Let $a=2, s=[s]\geq2, q>s+\frac12 $, then the solution to (\ref{B})  with the initial data $u_0\in H^q$  exists globally in time.
   \end{lemm}
   \begin{lemm}\cite{He-Yin-new}\label{1-global-11}
  Suppose $s>\frac32.$  If $ a=2$ and $q\geq 2s,$  then the solution to (\ref{B})  with the initial data $u_0\in H^q$  exists globally in time.
   \end{lemm}
   \begin{lemm}\cite{He-Yin-new}\label{2global1}
  Suppose $s=[s]=k\geq 2,\, a=2, \kappa\geq 0,\,  q>s+\frac12$ and  $q_1$ satisfies the  following condition  \begin{align}\label{condition3}
\frac12< q_1\leq q-1\leq q_1+2s-2.
     \end{align}
    Given the initial data $(u_0,\rho_0)\in H^q\times  H^{q_1}$, then the solution $(u,\rho)$ to  (\ref{1}) exists globally in time, namely, $(u,\rho)\in C([0,\infty); H^q\times  H^{q_1}).$
 \end{lemm}
 \begin{lemm}\cite{He-Yin-new}\label{2global2}
  Suppose $a=2, \, \kappa\geq 0,\,  s>\frac32$,\,  $q\geq 2s$  and  $q_1$ satisfies the  condition (\ref{condition3}).
  Given the initial data $(u_0,\rho_0)\in H^q\times  H^{q_1}$, then the solution $(u,\rho)$ to  (\ref{1}) exists globally in time, namely, $(u,\rho)\in C([0,\infty); H^q\times  H^{q_1}).$
 \end{lemm}

 For simplicity,  we will only consider the integer case that $s=[s]\geq2$ in the following part of our paper.
\section{Local Gevrey regularity and analyticity}
Now we can present a main theorem of our paper.
\begin{theo}\label{local-one-component}
 Let $\sigma\geq1$ and $q>s+\frac12.$ Assume that $u_0\in G^1_{\sigma,q}(\mathbb R).$ Then for every $0<\delta<1,$
  there exists  a $T_0$ such that  (\ref{B}) has a unique solution $u$ which is holomorphic  in
  $|t|<\frac{T_0(1-\delta)^\sigma}{2^\sigma-1}$ with values in $G^\delta_{\sigma, q}(\mathbb R).$
  Moreover, $T_0 \approx \frac1{\|u_0\|_{G^1_{\sigma,q}(\mathbb R)}}.$
\end{theo}
The one-component equation (\ref{B}) can be rewritten as (see \cite{He-Yin-new})
\begin{align}\label{equ.u}\begin{array}{lll}
u_t+uu_x=K(u,u)\\
\ \ \ \ \sim(1-\partial_x^2)^{-s}\Big(\partial_x^{2s-1}(u_x^2)+\partial_x^{2s-3}(u_{xx}^2)+...\\
~~~~~~~~~~~~~~~~~~~~~~~~~~+\partial_x[(\partial_x^su)^2]+\partial_x^{2s-3}(u_x^2)+...
+\partial_x[(\partial_x^{s-1}u)^2]
+...+\partial_x(u^2)\Big).
\end{array}
\end{align}
 In order to use  Lemma \ref{ACK-Luo}, we rewrite it as
\begin{align*}
   \left\{
   \begin{array}{lll}
      u_t=F(u):=-u\partial_xu+K(u,u),\\
      u|_{t=0}=u_0.
      \end{array}
   \right.
\end{align*}
For  a fixed $\sigma\geq1$ and $q>s+\frac12,$ Proposition \ref{embedding sequence}  ensures that
 $\{G^\delta_{\sigma,q}\}_{0<\delta<1}$ is a scale of decreasing Banach spaces.
 For any $0<\delta'<\delta,$ we need to estimate
 \begin{align*}
  \|F(u)\|_{G^{\delta'}_{\sigma,q}}\leq \frac12\|\partial_x(u^2)\|_{G^{\delta'}_{\sigma,q}}+\|K(u,u)\|_{G^{\delta'}_{\sigma,q}}.
    \end{align*}
    Note that  $\frac12\|\partial_x(u^2)\|_{G^{\delta'}_{\sigma,q}}\leq \frac{e^{-\sigma}\sigma^\sigma}{2(\delta-\delta')^\sigma}
   \|u^2\|_{G^{\delta}_{\sigma,q}} \leq C_q \frac{e^{-\sigma}\sigma^\sigma}{2(\delta-\delta')^\sigma}
   \|u\|^2_{G^{\delta}_{\sigma,q}}.$
   For any integers $1\leq i\leq s,0\leq j\leq s $ such that $(2i-1)+2j\leq2s+1,$ we have
   \begin{align*}
   \|(1-\partial_x^2)^{-s}\partial_x^{2i-1}[(\partial_x^ju)^2]\|_{G^{\delta'}_{\sigma,q}}
   \leq C_q\|(\partial_x^ju)^2\|_{G^{\delta'}_{\sigma,q-2s+2i-1}}\leq
   C_q\|(\partial_x^ju)^2\|_{G^{\delta'}_{\sigma,q-j}}\leq C_q\|u\|^2_{G^{\delta}_{\sigma,q}}.
   \end{align*}
   Hence
   \begin{align*}
   \|K(u,u)\|_{G^{\delta'}_{\sigma,q}}\leq C_{q,s}\|u\|^2_{G^{\delta}_{\sigma,q}}.
   \end{align*}
   Thus,  in view of $(\delta-\delta')^\sigma <1,$ we have
  $$ \|F(u)\|_{G^{\delta'}_{\sigma,q}}\leq  \frac{C_{q,s} (e^{-\sigma}\sigma^\sigma+2)}{2(\delta-\delta')^\sigma}
   \|u\|^2_{G^{\delta}_{\sigma,q}},$$
   which implies that $F$ satisfies  the condition (1) of Lemma \ref{ACK-Luo}. The similar computations yields
$$ \|F(u_0)\|_{G^{\delta}_{\sigma,q}}\leq  \frac{C_{q,s}(e^{-\sigma}\sigma^\sigma+2)}{2(1-\delta)^\sigma}
   \|u_0\|^2_{G^{1}_{\sigma,q}},$$
 which implies  that $F$ satisfies the condition (3) of Lemma \ref{ACK-Luo}  with
    $M=C_{q,s}(\frac{e^{-\sigma}\sigma^\sigma}{2}+1)\|u_0\|^2_{G^{1}_{\sigma,q}}.$
    It remains to verify that $F$ satisfies the condition (2) of Theorem \ref {ACK-Luo}. Assume that
    $\|u-u_0\|_{G^{\delta}_{\sigma,q}}\leq R$ and $\|v-v_0\|_{G^{\delta}_{\sigma,q}}\leq R.$  Applying  Propositions \ref{gradient estimate},
    we get
    \begin{align*}
     \|F(u)-F(v)\|_{G^{\delta'}_{\sigma,q}}&=\Big\|-\frac12\partial_x(u^2-v^2)+(K(u,u)-K(v,v))\Big\|_{G^{\delta'}_{\sigma,q}}\\
     &\leq \frac{e^{-\sigma}\sigma^\sigma}{2(\delta-\delta')^\sigma}
   \|(u+v)(u-v)\|_{G^{\delta}_{\sigma,q}}+\|K(u,u)-K(v,v)\|_{G^{\delta'}_{\sigma,q}},
       \end{align*}
       where
       \begin{align*}
       K(u,u)-K(v,v)\sim(1-\partial_x^2)^{-s}\Big(\partial_x^{2s-1}(u_x^2-v_x^2)
       +\partial_x^{2s-3}(u_{xx}^2-v_{xx}^2)+...+\partial_x[(\partial_x^su)^2-(\partial_x^sv)^2]+\\
+\partial_x^{2s-3}(u_x^2-v_x^2)+...+\partial_x[(\partial_x^{s-1}u-(\partial_x^{s-1}v)^2]
+...+\partial_x(u^2-v^2)\Big).
\end{align*}
Note that
 $$\|(u+v)(u-v)\|_{G^\delta_{\sigma,q}}\leq C_q\|u+v\|_{G^\delta_{\sigma,q}}\|u-v\|_{G^\delta_{\sigma,q}}
\leq C_q(2\|u_0\|_{G^\delta_{\sigma,q}}+2R)\|u-v\|_{G^\delta_{\sigma,q}}.$$
 On the other hand,  for any integers $1\leq i\leq s,0\leq j\leq s $ such that $(2i-1)+2j\leq2s+1,$ we obtain
   \begin{align*}
   \big\|(1-\partial_x^2)^{-s}\partial_x^{2i-1}[(\partial_x^ju)^2-(\partial_x^jv)^2]\big\|_{G^{\delta'}_{\sigma,q}}
   &\leq C_q\|(\partial_x^ju)^2-(\partial_x^jv)^2\|_{G^{\delta}_{\sigma,q-2s+2i-1}}\\
   &\leq C_q\|\partial_x^j(u+v)\partial_x^j(u-v)\|_{G^{\delta}_{\sigma,q-j}}\\
  & \leq C_q\|u+v\|_{G^{\delta}_{\sigma,q}}\|u-v\|_{G^{\delta}_{\sigma,q}}\\
   &\leq C_q(2\|u_0\|_{G^\delta_{\sigma,q}}+2R)\|u-v\|_{G^\delta_{\sigma,q}}.
   \end{align*}
   Therefore, we have  $\|K(u,u)-K(v,v)\|_{G^{\delta'}_{\sigma,q}}\leq C_{q,s}(2\|u_0\|_{G^\delta_{\sigma,q}}+2R)\|u-v\|_{G^\delta_{\sigma,q}},$ and
   \begin{align*}
    \|F(u)-F(v)\|_{G^{\delta'}_{\sigma,q}}\leq  \frac{C_{q,s}(e^{-\sigma}\sigma^\sigma+2)}{(\delta-\delta')^\sigma}
    (\|u_0\|_{G^1_{\sigma,q}}+R)\|u-v\|_{G^\delta_{\sigma,q}}.
   \end{align*}
Thus   $F$ satisfies the condition (2) of Lemma \ref{ACK-Luo}  with
$L=C_{q,s}(e^{-\sigma}\sigma^\sigma+2)(\|u_0\|_{G^1_{\sigma,q}}+R).$ Hence we obtain the local existence
result of (\ref{B}) with the Gevrey regularity or analyticity, and $T_0=\min \{\frac1{2^{2\sigma+4}L},\frac{(2^\sigma-1)R}{(2^\sigma-1)2^{2\sigma+3}LR+M}\}.$
By setting $R=\|u_0\|_{G^1_{\sigma,q}},$ we see that
 $L=2C_{q,s}(e^{-\sigma}\sigma^\sigma+2)\|u_0\|_{G^1_{\sigma,q}}$
 and $M\leq 2^{2\sigma+3}LR,$ and hence
 $$T_0=\min \{\frac1{2^{2\sigma+4}L},\frac{(2^\sigma-1)R}{(2^\sigma-1)2^{2\sigma+3}LR+M}\}=
 \frac1{2^{2\sigma+5}C_{q,s}(e^{-\sigma}\sigma^\sigma+2)\|u_0\|_{G^1_{\sigma,q}}}.\qed$$

We state another theorem to present the local Gevrey regularity and analyticity for the two-component system (\ref{1}).
 \begin{theo}\label{local-two-component}
 Let $\sigma \geq1$ and $q>s+\frac12.$ Assume that $u_0\in G^1_{\sigma,q}(\mathbb R)$
 and $\rho_0\in G^1_{\sigma,q_1}(\mathbb R).$ Then for  every $0<\delta<1,$
 there exists  a $T_0>0$ such that the two-component  system (\ref{1}) has a unique solution $(u,\rho)$ which is
 holomorphic in $|t|<\frac{T_0(1-\delta)^\sigma}{2^\sigma-1}$ with  values  in
 $G^\delta_{\sigma,q}(\mathbb R)\times G^\delta_{\sigma,q_1}(\mathbb R).$ Moreover, we can have
 $$ T_0\approx \frac 1{\|u_0\|_{G^1_{\sigma,q}}+\|\rho_0\|_{G^1_{\sigma,q_1}}+1}.$$
    \end{theo}
 \begin{proof}
  We  rewrite the two-component system (1.1) into the following form
  \begin{align*}
     \left\{\begin{array}{llll}
     z_t=F(z),\\
     z|_{t=0}=z_0,
       \end{array} \right.
     \end{align*}
     with $z=(u,\rho)^{\mathrm T}, z_0=(u_0,\rho_0)^{\mathrm T}$ and
     \begin{align}\label{Def-F}
          F(z)=\left(\begin{array}{lll} F_1(z)\\ F_2(z)\end{array}\right) :=\left(\begin{array}{lll}
        -\frac12\partial_x(u^2)+K(u,u)+(1-\partial_x^2)^{-s}(\alpha u_x-\frac\kappa2 \partial_x(\rho^2))\\
        -\partial_x(u\rho)+(2-a)u_x\rho\end{array}\right).
     \end{align}
     For fixed $\sigma\geq1$ and $s>\frac32,$ we set
      $X_\delta=G^\delta_{\sigma,q}(\mathbb R)\times G^\delta_{\sigma,q_1}(\mathbb R)$ and
      $$\|z\|_{\delta}=\|u\|_{G^\delta_{\sigma,q}}+\|\rho\|_{G^\delta_{\sigma,q_1}}.$$
      Proposition \ref{embedding sequence}  then ensures that $\{X_\delta\}_{0<\delta<1}$ is a scale of  decreasing Banach spaces.
      From the proof of Theorem \ref{local-one-component}, we have shown that for any $0<\delta'<\delta,$
      $$\| -\frac12\partial_x(u^2)+K(u,u)\|_{G^{\delta'}_{\sigma,q}}
      \leq  \frac{C_{q,s}(e^{-\sigma}\sigma^\sigma+2)}{2(\delta-\delta')^\sigma}
   \|u\|^2_{G^{\delta}_{\sigma,q}}.$$
   Note that $\|(1-\partial_x^2)^{-s}u_x\|_{G^{\delta'}_{\sigma,q}}\leq \|u\|_{G^{\delta}_{\sigma,q-2s+1}}\leq
   \|u\|_{G^{\delta}_{\sigma,q}}$ and
   \begin{align*}
    \|(1-\partial_x^2)^{-s}\partial_x(\rho^2)\|_{G^{\delta'}_{\sigma,q}}
    \leq \|\rho^2\|_{G^{\delta}_{\sigma,q-2s+1}}\leq \|\rho^2\|_{G^{\delta}_{\sigma,q_1}}\leq
    C_{q_1}\|\rho\|^2_{G^{\delta}_{\sigma,q_1}}.
   \end{align*}
   Hence, we  obtain
   $$\|F_1(z)\|_{G^{\delta'}_{\sigma,q}}\leq\frac{C_{q,s}(e^{-\sigma}\sigma^\sigma+2)}{2(\delta-\delta')^\sigma}
   \|u\|^2_{G^{\delta}_{\sigma,q}}+|\alpha|\cdot \|u\|_{G^{\delta}_{\sigma,q}}+\frac{|\kappa|}{2}C_{q_1}\|\rho\|^2_{G^{\delta}_{\sigma,q_1}}.$$
   On the other hand, considering the second component, we see
   \begin{align*}\begin{array}{lll}
      \|F_2(z)\|_{G^{\delta'}_{\sigma,q_1}}&=\|-\partial_x(u\rho)+(2-a)u_x\rho\|_{G^{\delta'}_{\sigma,q_1}}\\
     &\leq \frac{e^{-\sigma}\sigma^\sigma}{(\delta-\delta')^\sigma}
   \|u\rho\|_{G^{\delta}_{\sigma,q_1}}+|2-a|\cdot\|u_x\rho\|_{G^{\delta'}_{\sigma,q_1}}\\
   &\leq C_{q_1}\frac{e^{-\sigma}\sigma^\sigma}{(\delta-\delta')^\sigma}
   \|u\|_{G^{\delta}_{\sigma,q}}\|\rho\|_{G^{\delta}_{\sigma,q_1}}+|2-a|\cdot C_{q_1}\cdot
   \|u_x\|_{G^{\delta'}_{\sigma,q}}
   \|\rho\|_{G^{\delta'}_{\sigma,q_1}}\\
  & \leq  \frac{(|2-a|+1)C_{q_1}e^{-\sigma}\sigma^\sigma}{(\delta-\delta')^\sigma}
   \|u\|_{G^{\delta}_{\sigma,q}}\|\rho\|_{G^{\delta}_{\sigma,q_1}}.
  \end{array} \end{align*}
  Then  we can obtain
  $$\|F(z)\|_{\delta'}=\|F_1(z)\|_{G^{\delta'}_{\sigma,q}}
  +\|F_2(z)\|_{G^{\delta'}_{\sigma,q_1}} \leq
  C_{q,q_1,a,s,\alpha, \kappa}\frac{e^{-\sigma}\sigma^\sigma+2}{(\delta-\delta')^\sigma}
  (\|u\|_{G^{\delta}_{\sigma,q}}+\|\rho\|_{G^{\delta}_{\sigma,q_1}}+1)^2,$$
and $F$  satisfies the condition (1) of Lemma \ref{ACK-Luo}. By the similar computations, we obtain that
$$\|F(z_0)\|_{\delta} \leq
 C_{q,q_1,a,s,\alpha, \kappa}\,(\mathrm e^{-\sigma}\sigma^\sigma+2)
  (\|u_0\|_{G^{1}_{\sigma,q}}+\|\rho_0\|_{G^{1}_{\sigma,q_1}}+1)^2\,
  \frac1{{(1-\delta)^\sigma}},$$
which means that  $F$ satisfies the condition (3) of the Lemma \ref{ACK-Luo} with
$$
M=C_{q,q_1,a,s,\alpha, \kappa}\,(\mathrm e^{-\sigma}\sigma^\sigma+2)
  (\|z_0\|_1+1)^2.$$
  It remains to show that $F$ satisfies the condition (2) of Lemma \ref{ACK-Luo}.
   Assume that
   $$\|z_1-z_0\|_{\delta}\leq R, ~~ \|z_2-z_0\|_{\delta}\leq R,$$
   one can obtain
   \begin{align*}
   \begin{array}{lll}
   \|F_1(z_1)-F_1(z_2)\|_{G^{\delta'}_{\sigma,q}}\\
  \  \leq \|-\frac12\partial_x(u^2_1-u^2_2)+(K(u_1,u_1)-K(u_2,u_2))\|_{G^{\delta'}_{\sigma,q}}
   +\|(1-\partial_x^2)^{-s}\big(\alpha(u_1-u_2)_x-\frac\kappa2(\rho_1^2-\rho_2^2)_x\big)\|_{G^{\delta'}_{\sigma,q}}\\
  \  \leq  \frac{C_{q,s}(e^{-\sigma}\sigma^\sigma+2)}{(\delta-\delta')^\sigma}
    (\|u_0\|_{G^1_{\sigma,q}}+R)\|u_1-u_2\|_{G^\delta_{\sigma,q}}\\
   ~~~~~~~~~~~~~~~~~~
   +|\alpha| \|u_1-u_2\|_{G^\delta_{\sigma,q-2s+1}}+
    \frac{|\kappa|}{2}\|(\rho_1+\rho_2)(\rho_1-\rho_2)\|_{G^\delta_{\sigma,q-2s+1}},
    \end{array}
   \end{align*}
   with
   \begin{align*}
    \|(\rho_1+\rho_2)(\rho_1-\rho_2)\|_{G^\delta_{\sigma,q-2s+1}}&\leq
   \|(\rho_1+\rho_2)(\rho_1-\rho_2)\|_{G^\delta_{\sigma,q_1}} \\
   &\leq C_{q_1}(\|\rho_1\|_{G^\delta_{\sigma,q_1}}\|\rho_2\|_{G^\delta_{\sigma,q_1}})\|\rho_1-\rho_2\|_{G^\delta_{\sigma,q_1}}\\
   &\leq C_{q_1}(2\|\rho_0\|_{G^\delta_{\sigma,q_1}}+2R)\|\rho_1-\rho_2\|_{G^\delta_{\sigma,q_1}}.
\end{align*}
It then follows that
\begin{align*}
\|F_1(z_1)-F_1(z_2)\|_{G^{\delta'}_{\sigma,q}}\leq \frac {C_{q,q_1,a,s,\alpha,\kappa}\,(\mathrm e^{-\sigma}\sigma^\sigma+2)
  (\|z_0\|_1+1+R)}{(\delta-\delta')^\sigma}\|z_1-z_2\|_\delta.
\end{align*}
On the other hand, we see
\begin{align*}
   &\|F_2(z_1)-F_2(z_1)\|_{G^{\delta'}_{\sigma,q_1}}\\
   &\ \ \ =\|-\partial_x(u_1\rho_1- u_2\rho_2)
   +(2-a)(u_{1x}\rho_1-u_{2x}\rho_2)\|{G^{\delta'}_{\sigma,q_1}}\\
   &\ \ \ \leq\frac{e^{-\sigma}\sigma^\sigma}{(\delta-\delta')^\sigma}\|u_1(\rho_1-\rho_2)
   +(u_1-u_2)\rho_2\|_{G^\delta_{\sigma,q_1}}
   +|2-a|\cdot \|(u_1-u_2)_x\rho_1+u_{2x}(\rho_1-\rho_2)\|_{G^{\delta'}_{\sigma,q_1}}.
\end{align*}
Since $q,q_1$ satisfies the condition (\ref{condition3}), according to  Lemma \ref{productlaws}, we have
\begin{align*}
   \|u_1(\rho_1-\rho_2)
   +(u_1-u_2)\rho_2\|_{G^\delta_{\sigma,q_1}}
   &\leq C_{q_1}\|u_1\|_{G^\delta_{\sigma,q}}\|\rho_1-\rho_2\|_{G^\delta_{\sigma,q_1}}
    +C_{q_1}\|u_1-u_2\|_{G^\delta_{\sigma,q}}\|\rho_2\|_{G^\delta_{\sigma,q_1}}\\
    &\leq  C_{q_1}(\|z_0\|_\delta+R)\|z_1-z_2\|_{\delta},
    \end{align*}
    and
 \begin{align*}
 \begin{array}{lll}
 \|(u_1-u_2)_x\rho_1+u_{2x}(\rho_1-\rho_2)\|_{G^{\delta'}_{\sigma,q_1}}\\
 \ \ \ \ \ \ \leq C_{q_1}\big(\|(u_1-u_2)_x\|_{G^{\delta'}_{\sigma,q_1}}\|\rho_1\|_{G^{\delta'}_{\sigma,q_1}}+
\|u_{2x}\|_{G^{\delta'}_{\sigma,q_1}}\|\rho_1-\rho_2\|_{G^{\delta'}_{\sigma,q_1}}\big)\\
\ \ \ \ \ \ \leq C_{q_1}\frac{e^{-\sigma}\sigma^\sigma}{(\delta-\delta')^\sigma}
\Big\{\|u_1-u_2\|_{G^{\delta}_{\sigma,q}}(\|\rho_0\|_{G^{\delta}_{\sigma,q_1}}+R)+
(\|u_{0}\|_{G^{\delta}_{\sigma,q}}+R)\|\rho_1-\rho_2\|_{G^{\delta}_{\sigma,q_1}}\Big\}\\
\ \ \ \ \ \ \leq C_{q_1}\frac{e^{-\sigma}\sigma^\sigma}{(\delta-\delta')^\sigma}
(\|z_0\|_\delta+R)\|z_1-z_2\|_{\delta}.
\end{array}
    \end{align*}
    and hence
    \begin{align*}
   \|F_2(z_1)-F_2(z_1)\|_{G^{\delta'}_{\sigma,q_1}}
   \leq C_{q_1}(1+|a-2|)\frac{e^{-\sigma}\sigma^\sigma}{(\delta-\delta')^\sigma}
(\|z_0\|_1+R)\|z_1-z_2\|_{\delta}.
   \end{align*}

    Therefore, we deduce  that
    \begin{align}\label{F-condition2}
          \|F(z_1)-F(z_2)\|_{\delta'}&\leq  \|F_1(z_1)-F_1(z_2)\|_{G^{\delta'}_{\sigma,q}}
          +\|F_2(z_1)-F_2(z_2)\|_{G^{\delta'}_{\sigma,q_1}}\nonumber\\
          &\leq \frac {C_{q,q_1,a,s,\alpha, \kappa}\,(\mathrm e^{-\sigma}\sigma^\sigma+2)
  (\|z_0\|_1+1+R)}{(\delta-\delta')^\sigma} \|z_1-z_2\|_\delta,
    \end{align}
     and  $F$ satisfies  the condition (2)  of Lemma \ref{ACK-Luo}  with $L=C_{q,q_1,a,s,\alpha,\kappa}\,(\mathrm e^{-\sigma}\sigma^\sigma+2)
  (\|z_0\|_1+1+R).$ Hence we obtain the local existence
result of (\ref{B}) with the Gevrey regularity or analyticity, and $$T_0=\min \{\frac1{2^{2\sigma+4}L},\frac{(2^\sigma-1)R}{(2^\sigma-1)2^{2\sigma+3}LR+M}\}.$$
Moreover, by setting  $R=\|z_0\|_1+1,$
     we can see $L=2C_{q,q_1,a,s,\alpha,\kappa}\,(\mathrm e^{-\sigma}\sigma^\sigma+2)(\|z_0\|_1+1)$ and  $M\leq 2^{2\sigma+3}LR.$ It then follows that
     \begin{align*}
       T_0
       =\frac1{2^{2\sigma+5}C_{q,q_1,a,s,\alpha, \kappa}(e^{-\sigma}\sigma^\sigma+2)(\|z_0\|_1+1)}.
     \end{align*}
     This  completes the proof of  Theorem \ref{local-two-component}.
   \end{proof}
    \section{Continuity of the data-to-solution map}
    In this section, we study the continuity of the data-to-solution map for initial data and solutions in Theorems \ref{local-one-component}--\ref{local-two-component}.
    We only prove this  for the two-component system (1.1).

     At first we introduce  a definition to explain what means the data-to-solution map is continuous  from
     $G^1_{\sigma,q}(\mathbb R)\times G^1_{\sigma,q_1}(\mathbb R)$ into the solution space.
     \begin{defi}
      Let $\sigma\geq1,q>s+\frac12$ and let $q_1$ satisfy the condition (\ref{condition3}). We say that the data-to-solution map
      $(u_0,\rho_0)\mapsto (u,\rho)$ of the system (1.1) is continuous if for a given initial datum
      $(u_0^\infty,\rho_0^\infty)\in G^1_{\sigma,q}\times G^1_{\sigma,q_1},$ there exists a
      $T=T(\|u_0^\infty\|_{G^1_{\sigma,q}} ,\|\rho_0^\infty\|_{G^1_{\sigma,q_1}})>0,$  such that for any sequence
       $(u_0^n,\rho_0^n)\in G^1_{\sigma,q}\times G^1_{\sigma,q_1}$ and
       $\|u_0^n-u_0^\infty\|_{G^1_{\sigma,q}} +\|\rho_0^n-\rho_0^\infty\|_{G^1_{\sigma,q_1}}\
       \xrightarrow {n\rightarrow \infty} 0,$ the corresponding  solutions $(u^n,\rho^n)$ of system (1.1) satisfy
       $\|z^n-z^\infty\|_{E_T}:=\|u^n-u^\infty\|_{E_{q,T}} +\|\rho^n-\rho^\infty\|_{E_{q_1,T}}
       \xrightarrow {n\rightarrow \infty} 0,$ where
       \begin{align*}
       \|f\|_{E_{q,T}} := \sup\limits_{|t|<\frac{T(1-\delta)^\sigma}{2^\sigma-1}}
       \Big(\|f(t)\|_{G^\delta_{\sigma,q}}(1-\delta)^\sigma \sqrt{1-\frac{|t|}{T(1-\delta)^\sigma}}\ \Big).
       \end{align*}
     \end{defi}
     Also, we need to introduce the following lemma.
     \begin{lemm}\cite{Luo-Y}\label{Luo3.7}
       Let $\sigma\geq1.$ For every  $a>0,u\in E_a,0<\delta<1$ and $0\leq t <\frac{a(1-\delta)^\sigma}{2^\sigma-1}$
       we have$$  \int^t_0\frac{\|u(\tau)\|_{\delta(\tau)}}{(\delta(\tau)-\delta)^\sigma}\mathrm d \tau
      \leq\frac{a2^{2\sigma+3}\|u\|_{E_a}}{(1-\delta)^\sigma}\sqrt{\frac{a(1-\delta)^\sigma}{a(1-\delta)^\sigma-t}}\, ,
       $$
       where $\delta(t)=\frac12(1+\delta)+(\frac12)^{2+\frac1\delta}\Big\{[(1-\delta)^\sigma-\frac ta]^\frac1\sigma
       -[(1-\delta)^\sigma+(2^{\sigma+1}-1)\frac ta]^\frac1\sigma\Big\}\,\in (\delta,1).$
     \end{lemm}
     Now we can state the main theorem of  this section.
     \begin{theo}\label{solution-continuity}
       Let $\sigma\geq1, q>s+\frac12$ and  let $q_1$ satisfy the condition (\ref{condition3}). Assume
       $(u_0,\rho)\in G^1_{\sigma,q}(\mathbb R)\times G^1_{\sigma,q_1}(\mathbb R).$
       Then the data-to-solution map $(u_0,\rho_0)\mapsto (u,\rho)$ of the system (1.1) is continuous from
       $G^1_{\sigma,q}\times G^1_{\sigma,q_1}$  into  the  solution space.
     \end{theo}

     \begin{proof}
     Without loss of generality, we may assume that $t\geq0.$ As in the proof of Theorem \ref{local-two-component}, we use the
     notations $z^n=(u^n,\rho^n)^{\mathrm T}, z_0^n=(u_0^n,\rho_0^n)^{\mathrm T},$
     $\|z^n\|_\delta=\|u^n\|_{G^\delta_{\sigma,q}}+\|\rho^n\|_{G^\delta_{\sigma,q_1}}$ and
     $\|z\|_{E_T}= \|u^n\|_{E_{q,T}}+\|\rho^n\|_{E_{q_1,T}}.$ Define that
     \begin{align*}\label{T-def}
        T^\infty =\frac1{2^{2\sigma+5}C_{q,q_1,s,\alpha,\kappa}\,(\mathrm e^{-\sigma}\sigma^\sigma+2)(\|z_0^\infty\|_1+1)},\ \ \
        T^n =\frac1{2^{2\sigma+5}C_{q,q_1,s,\alpha,\kappa}\,(\mathrm e^{-\sigma}\sigma^\sigma+2)(\|z_0^n\|_1+1)},
     \end{align*}
    where $C_{q,q_1,s,\alpha,\kappa}$ is given in  (\ref{F-condition2}).
    Since  $\|z_0^n-z_0^\infty\|_1 \xrightarrow{n\rightarrow\infty} 0,$ it follows that there exists a constant $N$ such that
    if $n\geq N,$ we can have
    \begin{align}
      \|z_0^n\|_1\leq\|z_0^\infty\|_1+1.
    \end{align}
    By setting
    \begin{align}
     T =\frac1{2^{2\sigma+5}C_{q,q_1,s,\alpha,\kappa}\,(\mathrm e^{-\sigma}\sigma^\sigma+2)(\|z_0^n\|_1+2)},
    \end{align}
    we deduce from (\ref{T-def}) that  $T<\min\{T^n,T^\infty\}$ for  any $n\geq N.$ As in the proof of Theorem \ref{local-two-component},
    we see that $T^n$ and $T^\infty$ are the existence time of the solutions $z^n$ and $z^\infty$ corresponding to
    $z_0^n$ and $z_0^\infty$ respectively. Thus, we can see, for any $n\geq N,$
     \begin{align*}
        z^\infty(t,x)=z_0^\infty(x)+\int_0^tF(z^\infty(t,x))\mathrm d\tau,\ 0\leq t<\frac{T(1-\delta)^\sigma}{2^\sigma-1},\\
        z^n(t,x)=z_0^n(x)+\int_0^tF(z^n(t,x))\mathrm d\tau,\ 0\leq t<\frac{T(1-\delta)^\sigma}{2^\sigma-1},
     \end{align*}
     where $F$ is given in (\ref{Def-F}). Therefore, for any $0\leq t<\frac{T(1-\delta)^\delta}{2^\sigma-1}$ and $0<\delta<1,$
     we have
     \begin{align}\label{zuocha}
      \| z^n(t)-z^\infty(t)\|_\delta\leq \| z^n_0-z^\infty_0\|_\delta
      +\int_0^t\|F(z^n(\tau))-F(z^\infty(\tau))\|_\delta\mathrm d \tau.
     \end{align}
     Define that $\delta(t)=\frac12(1+\delta)+(\frac12)^{2+\frac1\delta}\big\{[(1-\delta)^\sigma-\frac t T]^{\frac1\sigma}-[(1-\delta)^\sigma +(2^{\sigma+1}-1)\frac t T]^{\frac1\sigma}\big\}.$
     By virtue of  Lemma \ref{Luo3.7} with $a=T,$  we see that $\delta<\delta(\tau)<1.$ Taking advantage of (\ref{F-condition2}), we obtain
     \begin{align}
        \|F(z^n(\tau))-F(z^\infty(\tau))\|_\delta\leq L\frac{\|z^n(\tau)-z^\infty(\tau)\|_{\delta(\tau)}}{(\delta(\tau)-\delta)^\sigma}
     \end{align}
     with $L=2C_{q,q_1,a,s,\alpha,\kappa}\,(\mathrm e^{-\sigma}\sigma^\sigma+2)(\|z_0\|_1+1).$
     Plugging it into (\ref{zuocha}) yields that
     \begin{align*}
      \| z^n(t)-z^\infty(t)\|_\delta\leq \| z^n_0-z^\infty_0\|_\delta
      +L\int_0^t\frac{\|z^n(\tau)-z^\infty(\tau)\|_{\delta(\tau)}}{(\delta(\tau)-\delta)^\sigma}\mathrm d \tau.
     \end{align*}
 Applying  Lemma \ref{Luo3.7},  we deduce that
\begin{align*}
      \| z^n(t)-z^\infty(t)\|_\delta\leq \| z^n_0-z^\infty_0\|_\delta
      +L\frac{T2^{2\sigma+3}\|z^n-z^\infty\|_{E_T}}{(\delta(1-\delta)^\sigma}
      \sqrt{\frac{T(1-\delta)^\sigma}{T(1-\delta)^\sigma-t}}\, .
     \end{align*}
   Noting that $LT2^{2\sigma+3}<\frac12,$  we can obtain
     \begin{align*}
      \| z^n(t)-z^\infty(t)\|_\delta\leq \| z^n_0-z^\infty_0\|_\delta
      +\frac1{2(1-\delta)^\delta}\|z^n-z^\infty\|_{E_T}
      \sqrt{\frac{T(1-\delta)^\sigma}{T(1-\delta)^\sigma-t}}\, ,
     \end{align*}
     which leads to
     \begin{align*}
     \| z^n(t)-z^\infty(t)\|_\delta(1-\delta)^\sigma\sqrt{1-\frac t{T(1-\delta)^\sigma}}&\leq
     \| z^n_0-z^\infty_0\|_\delta(1-\delta)^\sigma\sqrt{1-\frac t{T(1-\delta)^\sigma}}\
     + \frac12 \|z^n-z^\infty\|_{E_T}\\
    &\leq \|z_0^n-z_0^\infty\|_1+\frac12 \|z^n-z^\infty\|_{E_T}.
     \end{align*}
     Note that the right-hand side of the above inequality is independent of $t$ and $\delta.$
     By taking the supremum over $0<\delta<1, 0<t<\frac{T(1-\delta)^\sigma}{2^\sigma-1},$ we obtain
     \begin{align*}
      \|z^n-z^\infty\|_{E_T}\leq \|z_0^n-z_0^\infty\|_1+\frac12 \|z^n-z^\infty\|_{E_T},
     \end{align*}
     or
     \begin{align*}
      \|z^n-z^\infty\|_{E_T}\leq 2\|z_0^n-z_0^\infty\|_1.
     \end{align*}
    This inequality holds true for any $n\geq N$ and completes the proof of Theorem \ref{solution-continuity}.
     \end{proof}
\section{Global  Gevrey regularity and analyticity}
We firstly introduce a lemma which is crucial to deal with the convection term of the system (1.1). The idea comes
from \cite{Lev-O}, but we  release the restriction on $r$.
\begin{lemm}\label{commutator}
 Let $\delta\geq0, \, \sigma\geq1$ and $r>1+\frac d 2.$ Let $u\in  G^{\delta}_{\sigma,r+\frac1{2\sigma}}(\mathbb R^d)$ and $w\in G^{\delta}_{\sigma,r+\frac1{2\sigma}}(\mathbb R^d)$. Then one has the estimate
 \begin{align*}
   |\langle A^r e^{\delta A^{\frac1\sigma}}(u\cdot \nabla w), A^r e^{\delta A^{\frac1\sigma}} w\rangle|
   \leq C\|A^ru\| \|A^rw\|^2+C\delta \big(\|A^r e^{\delta A^{\frac1\sigma}}u\|
   \|A^{r+\frac1{2\sigma}} e^{\delta A^{\frac1\sigma}}w\|^2+\\
   +\|A^{r+\frac1{2\sigma}} e^{\delta A^{\frac1\sigma}}u\|\|A^r e^{\delta A^{\frac1\sigma}}w\|
   \|A^{r+\frac1{2\sigma}} e^{\delta A^{\frac1\sigma}}w\|
   \big),
  \end{align*}
  where $\langle\cdot,\cdot\rangle$ denotes the scalar product of $L^2(\mathbb R^d),$ and $\|\cdot\|:=\|\cdot\|_{L^2(\mathbb R^d)}.$
  \end{lemm}
  It is helpful to introduce  two lemmas.
  \begin{lemm}\label{cha-gu-ji}
   Let $r\geq1,\sigma\geq1$ and $\delta\geq0.$ Then for  any real $\xi,\eta,$  there holds
   \begin{align}\label{difference-estimate}
   &\big|(1+\xi^2)^{\frac r2}\mathrm e^{\delta(1+\xi^2)^{\frac1{2\sigma}}}-(1+\eta^2)^{\frac r2}
   \mathrm e^{\delta(1+\eta^2)^{\frac1{2\sigma}}}\big|\nonumber\\
   &\ \ \ \ \  \leq C_r \big|\xi-\eta\big|\Big\{(1+|\xi-\eta|^2)^{\frac{r-1}2}+(1+|\eta|^2)^{\frac{r-1}2}\nonumber\\
   &\ \ \ \ \  \ \ \ \ \ \ \  \  \ \ \  \ \ \ \ \ \ \ \ \ +\delta\big[
   (1+|\xi-\eta|^2)^{\frac{r-1}2+\frac1{2\sigma}}+(1+|\eta|^2)^{\frac{r-1}2+\frac1{2\sigma}}
   \big]\mathrm e^{\delta(1+|\xi-\eta|^2)^\frac1{2\sigma}}\mathrm e^{\delta(1+|\eta|^2)^\frac1{2\sigma}}
   \Big\}.
   \end{align}
   \end{lemm}
\begin{proof}
Without loss of generality, take $\xi>\eta\geq0.$
Set $f(\theta):=(1+\theta^2)^{\frac r2}\mathrm e^{\delta(1+\theta^2)^{\frac1{2\sigma}}}.$ By the mean value
theorem, we see
\begin{align*}
   (1+\xi^2)^{\frac r2}\mathrm e^{\delta(1+\xi^2)^{\frac1{2\sigma}}}-(1+\eta^2)^{\frac r2}
   \mathrm e^{\delta(1+\eta^2)^{\frac1{2\sigma}}}\leq (\xi-\eta)\sup\limits_{\theta\in[\eta,\xi]}|f'(\theta)|.
\end{align*}
Computing  $f'$ and using  the fact that  $e^y\leq 1+y\mathrm e^y$ for $y\geq0$ yield
\begin{align}\label{supremum}
  f'(\theta)&=r\theta(1+\theta^2)^{\frac r2-1}\mathrm e^{\delta(1+\theta^2)^{\frac1{2\sigma}}}+\frac{\delta}{\sigma}
  \theta (1+\theta^2)^{\frac r2+\frac1{2\sigma}-1}\mathrm e^{\delta(1+\theta^2)^{\frac1{2\sigma}}}\nonumber\\
  &\leq r(1+\theta^2)^{\frac{r-1}2}
  \big[1+\delta(1+\theta^2)^{\frac1{2\sigma}}\mathrm e ^{\delta(1+\theta^2)^{\frac1{2\sigma}}}\big]+
 \frac{\delta}{\sigma}(1+\theta^2)^{\frac {r-1}2+\frac1{2\sigma}}
 \mathrm e^{\delta(1+\theta^2)^{\frac1{2\sigma}}} \nonumber\\
&= r(1+\theta^2)^{\frac{r-1}2}+\delta(r+\frac1\sigma)(1+\theta^2)^{\frac {r-1}2+\frac1{2\sigma}}\mathrm e^{\delta(1+\theta^2)^{\frac1{2\sigma}}}.
\end{align}
It is easy to verify that $f'$ is a monotonically increasing function for $r\geq 1,$ and hence the supremum in (\ref{supremum}) is attained when $\theta=\xi.$ For arbitrary non-negative $\xi$ and $\eta,$ we have
\begin{align}\label{estimate1}
  &\big| (1+\xi^2)^{\frac r2}\mathrm e^{\delta(1+\xi^2)^{\frac1{2\sigma}}}-(1+\eta^2)^{\frac r2}
   \mathrm e^{\delta(1+\eta^2)^{\frac1{2\sigma}}}\big|\nonumber\\
  &\ \  \leq \big|\xi-\eta\big|\Big(
   r\big[(1+\xi^2)^{\frac{r-1}2}+ (1+\eta^2)^{\frac{r-1}2}\big]\nonumber\\
   &~~~~~~~~~~~~~~~~~~~~+\delta(r+\frac1\sigma)\big[(1+\xi^2)^{\frac {r-1}2+\frac1{2\sigma}}
   \mathrm e^{\delta(1+\xi^2)^{\frac1{2\sigma}}}
  +(1+\eta^2)^{\frac {r-1}2+\frac1{2\sigma}}\mathrm e^{\delta(1+\eta^2)^{\frac1{2\sigma}}}\big]\Big).
   \end{align}
     Note that for any  $\xi,\eta,a,b\in\mathbb R,$
   \begin{align}\label{estimate2}
   |\xi|^\rho\leq\left\{\begin{array}{lll}
    |\xi-\eta|^\rho+|\eta|^\rho,\ \ \ \  \ \ \  \ \ \ \ \ \  &\mathrm{when} \ \ \ \rho\in(0,1],\\
    2^{\rho-1}( |\xi-\eta|^\rho+|\eta|^\rho),\ \ \ \  \ \ \ &\mathrm{when} \ \ \ \rho >1,
    \end{array}\right.
   \end{align}
   and
   \begin{align}\label{ji-ben-bu-deng-shi}
    (1+(a+b)^2)^\frac12 \leq (1+a^2)^\frac12+(1+b^2)^\frac12.
   \end{align}
    Since $\sigma\geq1,$ it follows that
    \begin{align*}
(1+\xi^2)^{\frac1{2\sigma}}\leq \big((1+|\xi-\eta|^2)^\frac12+(1+|\eta|^2)^\frac12\big)^\frac1\sigma
 \leq(1+|\xi-\eta|^2)^\frac1{2\sigma}+(1+|\eta|^2)^\frac1{2\sigma}.
    \end{align*}
   which leads to
   \begin{align}\label{estimate3}
 \mathrm e^{\delta(1+\xi^2)^{\frac1{2\sigma}}}\leq \mathrm e^{\delta(1+|\xi-\eta|^2)^{\frac1{2\sigma}}}
 \mathrm e^{\delta(1+\eta^2)^{\frac1{2\sigma}}}
 \end{align}
 Combining the estimates (\ref{estimate1}), (\ref{estimate2}) and (\ref{estimate3}), we obtain (\ref{difference-estimate})
  and  complete the proof of Lemma \ref{cha-gu-ji}.
\end{proof}

Now we introduce a lemma to deal with the interpolation of the Sobolev spaces and the Gervey spaces. The proof
of this lemma is similar to that of  Lemma 8 in \cite{O-Titi}.
\begin{lemm}\label{interpolation-S-G lemma} For any $\delta\geq0,\, \sigma \geq1,\, l>0$ and $r\in\mathbb R,$ the following
estimate holds true:
 \begin{align}\label{interpolation-S-G}
   \|u\|_{ G^\delta_{\sigma,r}}\leq \sqrt{\mathrm e} \|u\|_{H^r}+(2\delta)^{\frac l2} \|u\|_{ G^\delta_{\sigma,r+\frac l{2\sigma}}}.
 \end{align}
 \end{lemm}
 \begin{proof}
 Noting that $\mathrm e ^y\leq \mathrm e+y^l\mathrm e^y$  for any $l>0, y\geq0$, we have
 \begin{align*}
 \|u\|^2_{ G^\delta_{\sigma,r}}&=\int (1+|\xi|^2)^r\mathrm e^{2\delta(1+|\xi|^2)^\frac1{2\sigma}}|\hat u(\xi)|^2
 \mathrm d \xi\\
 &\leq \int (1+|\xi|^2)^r\big( \mathrm e+(2\delta)^l(1+|\xi|^2)^\frac l{2\sigma}
 \mathrm e^{2\delta(1+|\xi|^2)^\frac1{2\sigma}}\big)|\hat u(\xi)|^2\mathrm d \xi\\
 &=\mathrm e\int (1+|\xi|^2)^r|u(\xi)|^2
 \mathrm d \xi +(2\delta)^l\int (1+|\xi|^2)^{r+\frac l{2\sigma}}
 \mathrm e^{2\delta(1+|\xi|^2)^\frac1{2\sigma}}|\hat u(\xi)|^2 \mathrm d \xi\\
 &=\mathrm e \|u\|^2_{H^r}+(2\delta)^l \|u\|^2_{ G^\delta_{\sigma,r+\frac l{2\sigma}}},
 \end{align*}
 which leads to (\ref{interpolation-S-G}).
 \end{proof}

{\bf Proof of Lemma \ref{commutator}.} The idea comes from \cite{Lev-O}. For simplicity, we only consider the case $d=1.$ Also, we  omit the subscript $\mathbb R$ and $\mathrm d \xi \mathrm d \eta $ of the integrands if there is no ambiguity.
Write
\begin{align*}
\langle A^r e^{\delta A^{\frac1\sigma}}(u\partial_xw), A^r e^{\delta A^{\frac1\sigma}} w\rangle=
\langle A^r e^{\delta A^{\frac1\sigma}}(u\partial_xw), A^r e^{\delta A^{\frac1\sigma}} w\rangle&-
\langle u\partial_x A^r e^{\delta A^{\frac1\sigma}} w, A^r e^{\delta A^{\frac1\sigma}} w\rangle\\
&\ \ +
\langle  u\partial_x A^r e^{\delta A^{\frac1\sigma}} w, A^r e^{\delta A^{\frac1\sigma}} w\rangle.
\end{align*}
Note that
\begin{align}\label{I0}
\big|\langle  u\partial_x A^r e^{\delta A^{\frac1\sigma}} w, A^r e^{\delta A^{\frac1\sigma}} w\rangle\big|
&=\big|\int u\big(\partial_x A^r e^{\delta A^{\frac1\sigma}} w\big) A^r e^{\delta A^{\frac1\sigma}} w\big|\nonumber\\
&=\big|-\frac12 \int u_x (A^r e^{\delta A^{\frac1\sigma}} w)^2\big|\nonumber\\
&\leq \frac12\|u_x\|_{L^\infty} \|A^r e^{\delta A^{\frac1\sigma}} w\|^2\nonumber\\
&\leq C_r\|u\|_{H^r}(\|w\|^2_{H^r}+\delta \|w\|^2_{ G^\delta_{\sigma,r+\frac 1{2\sigma}}}).
\end{align}
The inequality on the last line is due to Lemma \ref{interpolation-S-G lemma} with $l=1$ and the embedding
$H^{r-1}\hookrightarrow L^\infty.$\\
Denote $\phi=A^r\mathrm e^{\delta A^\frac1\sigma}w.$ Next, we need to find the bound of
\begin{align*}
 I&:= \langle A^r e^{\delta A^{\frac1\sigma}}(u\partial_xw), A^r e^{\delta A^{\frac1\sigma}} w\rangle-
\langle u\partial_x A^r e^{\delta A^{\frac1\sigma}} w, A^r e^{\delta A^{\frac1\sigma}} w\rangle\\
&\ =\langle u\partial_xw, A^r e^{\delta A^{\frac1\sigma}}\phi\rangle-
\langle u\partial_x A^r e^{\delta A^{\frac1\sigma}} w, \phi\rangle.
\end{align*}
By Planchel's  identity, we have
\begin{align*}
\langle u\partial_xw, A^r e^{\delta A^{\frac1\sigma}}\phi\rangle&=\int_{\mathbb R}(u\partial_xw) (A^r e^{\delta A^{\frac1\sigma}}\phi)\mathrm d x\\
&=i\int_{\mathbb R}\bar{\hat\phi}(\xi)(1+|\xi^2|)^{\frac r2}
\mathrm e^{\delta(1+|\xi|^2)^{\frac1{2\sigma}}}\int_{\mathbb R}
\hat u(\xi-\eta)\cdot \eta \hat w(\eta)\mathrm d\eta\mathrm d \xi\\
&=i\iint\bar{\hat\phi}(\xi)
\hat u(\xi-\eta)\hat w(\eta) \,\eta\, (1+|\xi|^2)^{\frac r2}
\mathrm e^{\delta(1+|\xi|^2)^{\frac1{2\sigma}}},
\\ \langle u\partial_x A^r e^{\delta A^{\frac1\sigma}} w, \phi\rangle
&=i\iint\bar{\hat\phi}(\xi)
\hat u(\xi-\eta)\hat w(\eta)\,\eta\, (1+|\eta|^2)^{\frac r2}
\mathrm e^{\delta(1+|\eta|^2)^{\frac1{2\sigma}}},
\end{align*}
where $\bar{\hat\phi}$ denotes the complex conjugate of the Fourier transformation of $\phi$. Applying Lemma \ref{cha-gu-ji} to yield
\begin{align}\label{I-estimate}
|I|\leq\iint |\bar{\hat\phi}(\xi)|\,
|\hat u(\xi-\eta)|\, |\hat w(\eta)|\,|\eta|\, \big|(1+|\xi|^2)^{\frac r2}
\mathrm e^{\delta(1+|\xi|^2)^{\frac1{2\sigma}}}-(1+|\eta|^2)^{\frac r2}
\mathrm e^{\delta(1+|\eta|^2)^{\frac1{2\sigma}}}\big|\nonumber\\
\leq C_r\iint |\bar{\hat\phi}(\xi)|\,
|\hat u(\xi-\eta)|\, |\hat w(\eta)|\,|\eta|\, |\xi-\eta|\,\Big\{(1+|\xi-\eta|^2)^{\frac{r-1}2}+(1+|\eta|^2)^{\frac{r-1}2}\nonumber\\
  +\delta\big[
   (1+|\xi-\eta|^2)^{\frac{r-1}2+\frac1{2\sigma}}+(1+|\eta|^2)^{\frac{r-1}2+\frac1{2\sigma}}
   \big]\mathrm e^{\delta(1+|\xi-\eta|^2)^\frac1{2\sigma}}\mathrm e^{\delta(1+|\eta|^2)^\frac1{2\sigma}}
   \Big\}.
\end{align}
By the definition of $\phi$, and the fact that $\mathrm e^y\leq1+y\mathrm e^y$ for $y\geq0,$ we obtain
\begin{align}\label{phi-decomposition}
  |\bar{\hat\phi}(\xi)|&=(1+|\xi|^2)^{\frac r2}\mathrm e^{\delta(1+|\xi|^2)^{\frac1\sigma}}|\bar{\hat w}(\xi)|\nonumber\\
  &\leq (1+|\xi|^2)^{\frac r2}\big( 1+\delta(1+|\xi|^2)^{\frac1\sigma}\mathrm e^{\delta(1+|\xi|^2)^{\frac1\sigma}}\big)\,|\bar{\hat w}(\xi)|\nonumber\\
  &=(1+|\xi|^2)^{\frac r2}|\bar{\hat w}(\xi)|+\delta(1+|\xi|^2)^{\frac1\sigma}|\bar{\hat \phi}(\xi)|.
\end{align}
Combining the estimate (\ref{I-estimate}) and (\ref{phi-decomposition}), we can divide  (\ref{I-estimate}) into four
parts
\begin{align}\label{two-parts}
 |I|&\leq C_r\iint (1+|\xi|^2)^{\frac r2}|\bar{\hat w}(\xi)|\,
|\hat u(\xi-\eta)|\, |\hat w(\eta)|\,|\eta|\, |\xi-\eta|\,\Big\{(1+|\xi-\eta|^2)^{\frac{r-1}2}+(1+|\eta|^2)^{\frac{r-1}2}\Big\}\mathrm d \eta\mathrm d \xi\nonumber\\
 &\ \ \ \ \ \ \  \  +C_r\delta\!\!\iint |\bar{\hat \phi}(\xi)|\,
|\hat u(\xi-\eta)|\, |\hat w(\eta)|\,|\eta|\, |\xi-\eta|\nonumber\\
 &\ \ \ \ \ \ \  \  \ \ \ \ \ \ \ \ \ \ \ \  \times\Big\{
   (1+|\xi-\eta|^2)^{\frac{r-1}2+\frac1{2\sigma}}+(1+|\eta|^2)^{\frac{r-1}2+\frac1{2\sigma}}
   \Big\}\mathrm e^{\delta(1+|\xi-\eta|^2)^\frac1{2\sigma}}
   \mathrm e^{\delta(1+|\eta|^2)^\frac1{2\sigma}}\mathrm d \eta\mathrm d \xi\nonumber\\
   &:= C_r(\, I_1+I_2\,)+ C_r\,\delta\cdot(I_3+I_4).
\end{align}
Firstly we consider the term
\begin{align}\label{I1}
I_1&=\iint (1+|\xi|^2)^{\frac r2}|\bar{\hat w}(\xi)|\,
|\hat u(\xi-\eta)|\, |\hat w(\eta)|\,|\eta|\, |\xi-\eta|\,(1+|\xi-\eta|^2)^{\frac{r-1}2}\mathrm d \eta\mathrm d \xi\nonumber\\
& \leq \int  |\hat w(\eta)|\,(1+|\eta|^2)^{\frac12} \Big(\int
(1+|\xi|^2)^{\frac r2}|\bar{\hat w}(\xi)|\,(1+|\xi-\eta|^2)^{\frac{r}2}
|\hat u(\xi-\eta)|\,\mathrm d \xi \Big)\mathrm d \eta\nonumber\\
& \leq \int  |\hat w(\eta)|\,(1+|\eta|^2)^{\frac r2}
 (1+|\eta|^2)^{\frac12-\frac r2}\mathrm d \eta\, \|A^ru\|\,\|A^rw\|\nonumber\\
& \leq C_r \|A^ru\|\,\|A^rw\|^2.
\end{align}
Similarly, after the transformation $\xi'=\xi, \eta'=\xi-\eta,$ we have
\begin{align}\label{I2}
I_2=\iint (1+|\xi|^2)^{\frac r2}|\bar{\hat w}(\xi)|\,
|\hat u(\xi-\eta)|\, |\hat w(\eta)|\,|\eta|\, |\xi-\eta|\,(1+|\eta|^2)^{\frac{r-1}2}\leq C_r\|A^ru\|\,\|A^rw\|^2.
\end{align}
Next, we consider the term $I_3$  in (\ref{two-parts}). Taking advantage  of (\ref{estimate2}) and (\ref{ji-ben-bu-deng-shi}), we obtain
\begin{align*}
(1+|\xi-\eta|^2)^\frac1{4\sigma} \leq\big((1+|\xi|^2)^\frac12+(1+|\eta|^2)^\frac12\big)^\frac1{2\sigma}\leq
(1+|\xi|^2)^\frac1{4\sigma}+(1+|\eta|^2)^\frac1{4\sigma},
\end{align*}
and hence
\begin{align}\label{I3}
I_3&=\iint |\bar{\hat \phi}(\xi)|\,
|\hat u(\xi-\eta)|\, |\hat w(\eta)|\,|\eta|\, |\xi-\eta|\,
   (1+|\xi-\eta|^2)^{\frac{r-1}2+\frac1{2\sigma}}\mathrm e^{\delta(1+|\xi-\eta|^2)^\frac1{2\sigma}}\mathrm e^{\delta(1+|\eta|^2)^\frac1{2\sigma}}\nonumber\\
   &\leq
   \iint |\bar{\hat \phi}(\xi)|\,|\hat u(\xi-\eta)|\, |\hat w(\eta)|\,(1+|\eta|^2)^\frac12
   (1+|\xi-\eta|^2)^{\frac{r}2+\frac1{4\sigma}}\nonumber\\
 &\ \ \ \ \ \ \ \ \ \ \ \ \ \ \ \ \ \ \ \ \ \ \ \ \ \ \ \  \ \ \ \ \ \ \  \ \ \times\Big[ (1+|\xi|^2)^\frac1{4\sigma}+(1+|\eta|^2)^\frac1{4\sigma}\Big]\mathrm e^{\delta(1+|\xi-\eta|^2)^\frac1{2\sigma}}
   \mathrm e^{\delta(1+|\eta|^2)^\frac1{2\sigma}}\nonumber\\
   &:= I_{31}+I_{32}.
\end{align}
It then follows that
\begin{align}\label{I31}
I_{31}&=\iint|\bar{\hat \phi}(\xi)|(1+|\xi|^2)^\frac1{4\sigma}|\hat u(\xi-\eta)|\,
   (1+|\xi-\eta|^2)^{\frac{r}2+\frac1{4\sigma}} \mathrm e^{\delta(1+|\xi-\eta|^2)^\frac1{2\sigma}}
   |\hat w(\eta)|\,(1+|\eta|^2)^\frac12| \mathrm e^{\delta(1+|\eta|^2)^\frac1{2\sigma}}\mathrm d \xi\mathrm d\eta\nonumber\\
 & \leq \int |\hat w(\eta)|\,(1+|\eta|^2)^\frac r2\mathrm e^{\delta(1+|\eta|^2)^\frac1{2\sigma}}
 (1+|\eta|^2)^{\frac12-\frac r2}\mathrm d\eta\,
   \|A^{r+\frac1{2\sigma}}\mathrm e^{\delta A^\frac1\sigma}\!u\|\,\|A^\frac1{2\sigma}\phi\|\nonumber\\
  &\leq \Big(\int
 (1+|\eta|^2)^{1-r}\mathrm d\eta\Big)^\frac12
 \|A^{r}\mathrm e^{\delta A^\frac1\sigma}\!w\|\,
   \|A^{r+\frac1{2\sigma}}\mathrm e^{\delta A^\frac1\sigma}\!u\|\,\|A^{r+\frac1{2\sigma}}\mathrm e^{\delta A^\frac1\sigma}\!w\|\nonumber\\
   &=C_r\|A^{r}\mathrm e^{\delta A^\frac1\sigma}\!w\|\,
   \|A^{r+\frac1{2\sigma}}\mathrm e^{\delta A^\frac1\sigma}\!u\|\,\|A^{r+\frac1{2\sigma}}\mathrm e^{\delta A^\frac1\sigma}\!w\|,
\end{align}
and
\begin{align}\label{I32}
I_{32}&=\iint|\bar{\hat \phi}(\xi)|\, |\hat u(\xi-\eta)|\,
   (1+|\xi-\eta|^2)^{\frac{r}2+\frac1{4\sigma}} \mathrm e^{\delta(1+|\xi-\eta|^2)^\frac1{2\sigma}}
   |\hat w(\eta)|\,(1+|\eta|^2)^{\frac12+\frac1{4\sigma}}| \mathrm e^{\delta(1+|\eta|^2)^\frac1{2\sigma}}\mathrm d \xi\mathrm d\eta\nonumber\\
 &\leq \int |\hat w(\eta)|\,(1+|\eta|^2)^{\frac r2+\frac1{4\sigma}}\mathrm e^{\delta(1+|\eta|^2)^\frac1{2\sigma}}
 (1+|\eta|^2)^{\frac12-\frac r2}\mathrm d\eta\,
   \|A^{r+\frac1{2\sigma}}\mathrm e^{\delta A^\frac1\sigma}\!u\|\,\|\phi\|\nonumber\\
  &\leq \Big(\int
 (1+|\eta|^2)^{1-r}\mathrm d\eta\Big)^\frac12
 \|A^{r+\frac1{2\sigma}}\mathrm e^{\delta A^\frac1\sigma}\!w\|\,
   \|A^{r+\frac1{2\sigma}}\mathrm e^{\delta A^\frac1\sigma}\!u\|\,\|A^{r}\mathrm e^{\delta A^\frac1\sigma}\!w\|\nonumber\\
   &=C_r\|A^{r}\mathrm e^{\delta A^\frac1\sigma}\!w\|\,
   \|A^{r+\frac1{2\sigma}}\mathrm e^{\delta A^\frac1\sigma}\!u\|\,\|A^{r+\frac1{2\sigma}}\mathrm e^{\delta A^\frac1\sigma}\!w\|.
\end{align}
Plugging (\ref{I31}) and (\ref{I32}) into (\ref{I3}) yields
\begin{align}\label{I3'}
  I_3\leq C_r\|A^{r}\mathrm e^{\delta A^\frac1\sigma}\!w\|\,
   \|A^{r+\frac1{2\sigma}}\mathrm e^{\delta A^\frac1\sigma}\!u\|\,\|A^{r+\frac1{2\sigma}}\mathrm e^{\delta A^\frac1\sigma}\!w\|.
\end{align}
Similarly, by the transformation $\xi'=\xi,\eta'=\xi-\eta,$ we obtain
\begin{align}\label{I4}
I_4&=\iint |\bar{\hat \phi}(\xi)|\,
|\hat u(\xi-\eta)|\, |\hat w(\eta)|\,|\eta|\, |\xi-\eta|(1+|\eta|^2)^{\frac{r-1}2+\frac1{2\sigma}}
   \mathrm e^{\delta(1+|\xi-\eta|^2)^\frac1{2\sigma}}\mathrm e^{\delta(1+|\eta|^2)^\frac1{2\sigma}}\nonumber\\
&\leq C_r\|A^{r}\mathrm e^{\delta A^\frac1\sigma}\!u\|\,
   \|A^{r+\frac1{2\sigma}}\mathrm e^{\delta A^\frac1\sigma}\!w\|^2.
\end{align}
According to (\ref{I0}), (\ref{I1}), (\ref{I2}), (\ref{I3'}) and (\ref{I4}), we complete the proof of Lemma \ref{commutator}.
\qed

Now we can state a main theorem of this section, which implies is the global Gevrey regularity and analyticity result of the equation
(\ref{B}).
\begin{theo}\label{one-component-global}
 Assuem $s=[s]\geq2, a=2$ and $\sigma\geq1.$ Let $u_{0}$ be in $ G_{\sigma}(\mathbb R).$   Then there exists  a unique
 global solution $u$ of (\ref{B}) in Gevrey class $\sigma,$ namely, for any $t\geq0,$ $u(t,\cdot)$ is of Gevrey class $\sigma.$
\end{theo}
\begin{proof}
   Here  we only derive  the global  a priori  bounds  on $u$ in the  time-dependent space $G^{\delta(t)}_{\sigma, q}$. One can use Fourier-Galerkin approximating method  to construct local solutions in $G^{\delta(t)}_{\sigma, q}$, and globalize the result by the later estimate (\ref{globalization}).  In the following, we  will find a $\delta(t)$ to keep the solution of Gervey class $\sigma$. Note that  if $u_0\in G^{\delta_0}_{\sigma,q}$ then
   $u_0\in G^{\delta_0-\varepsilon}_{\sigma,\infty}$ for any $\varepsilon>0.$ Without loss of generality, we may
   assume that $q>s+\frac12$ and  $u_0\in G^1_{\sigma,q}(\mathbb R).$
  Since  for  every $t\in[0,T)$
  \begin{align*}
  u_t=-uu_x+K(u,u),
  \end{align*}
  we then have
  \begin{align}\label{energy-estimate}
   \frac{\mathrm d}{\mathrm d t}\|u\|^2_{ G^{\delta(t)}_{\sigma,q}}
   &=\frac{\mathrm d}{\mathrm d t}
   \int (1+|\xi|^2)^q\mathrm e^{2\delta(t)(1+|\xi|^2)^\frac1{2\sigma}}
   \hat u(\xi)\bar{\hat u}(\xi)\mathrm d\xi\nonumber\\
   &=2\dot\delta(t)\int (1+|\xi|^2)^{q+\frac1{2\sigma}}
   \mathrm e^{2\delta(t)(1+|\xi|^2)^\frac1{2\sigma}}\hat u(\xi)\bar{\hat u}(\xi)\mathrm d \xi\nonumber\\
  &\ \ \ +2\mathfrak{Re} \int (1+|\xi|^2)^q\mathrm e^{2\delta(t)(1+|\xi|^2)^\frac1{2\sigma}}\big(-\widehat {uu_x}(\xi)+\widehat{K(u,u)}(\xi)\big)\bar{\hat u}(\xi)\mathrm d\xi,
  \end{align}
  where $\mathfrak{Re}$ denotes the real part of a complex number. Applying Lemma \ref{commutator} to obtain
  \begin{align}\label{control1}
 \big|\int (1+|\xi|^2)^q\mathrm e^{2\delta(t)(1+|\xi|^2)^\frac1{2\sigma}}\widehat {uu_x}(\xi)\bar{\hat u}(\xi)
  \mathrm d\xi\big|
  &= |\langle A^q e^{\delta A^{\frac1\sigma}}\!(u\partial_x u), A^q e^{\delta A^{\frac1\sigma}}\! u\rangle\nonumber|\\
  &\leq C_q \big(\,\|u\|^3_{H^q}+\delta \|u\|_{G^{\delta}_{\sigma,q}}\|u\|^2_{G^{\delta}_{\sigma,q+\frac1{2\sigma}}}\big).
  \end{align}
  To control  the term involved $K(u,u),$ for simplicity, we only control the first and the last term of the highest order in
  (\ref{equ.u}), namely,
  \begin{align*}
&\big|\int (1+|\xi|^2)^q\mathrm e^{2\delta(t)(1+|\xi|^2)^\frac1{2\sigma}}\mathcal{F}\big( (1-\partial_x^2)^{-s}\partial_x^{2s-1}(u_x^2)\big)(\xi)\bar{\hat u}(\xi)\mathrm d\xi\big|\\
&\ =\big| \langle A^q\mathrm e^{\delta A^\frac1\sigma}\big( (1-\partial_x^2)^{-s}\partial_x^{2s-1}(u_x^2)\big),
A^q\mathrm e^{\delta A^\frac1\sigma}u\rangle\big|\\
&\ \leq \| A^{q-2s}\mathrm e^{\delta A^\frac1\sigma}\partial_x^{2s-1}(u_x^2)\|_{L^2}\|A^q\mathrm e^{\delta A^\frac1\sigma}u\|_{L^2}\\
&\ \leq \|u^2_x\|_{ G^\delta_{\sigma,q-1}}\|u\|_{ G^\delta_{\sigma,q}}\\
&\ \leq C_q\|u\|^3_{ G^\delta_{\sigma,q}},
  \end{align*}
  and
  \begin{align*}
   &\big|\int (1+|\xi|^2)^q\mathrm e^{2\delta(t)(1+|\xi|^2)^\frac1{2\sigma}}\mathcal{F}\big( (1-\partial_x^2)^{-s}\partial_x[(\partial_x^su)^2]\big)(\xi)\bar{\hat u}(\xi)\mathrm d\xi\big|\\
&\ =\big| \langle A^q\mathrm e^{\delta A^\frac1\sigma}\big( (1-\partial_x^2)^{-s}\partial_x[(\partial_x^su)^2]\big),
A^q\mathrm e^{\delta A^\frac1\sigma}u\rangle\big|\\
&\ \leq  \| A^{q-2s}\mathrm e^{\delta A^\frac1\sigma}\partial_x[(\partial_x^su)^2]
\|_{L^2}\|A^q\mathrm e^{\delta A^\frac1\sigma}u\|_{L^2}\\
&\ \leq \|(\partial_x^su)^2\|_{ G^\delta_{\sigma,q-2s+1}}\|u\|_{ G^\delta_{\sigma,q}}\\
&\ \leq\|(\partial_x^su)^2\|_{ G^\delta_{\sigma,q-s}}\|u\|_{ G^\delta_{\sigma,q}}\\
&\ \leq C_{q,s}\|u\|^3_{ G^\delta_{\sigma,q}}.
  \end{align*}
  Therefore, using  Lemma \ref{interpolation-S-G lemma} with $l=\frac23,$ we have
  \begin{align}\label{control2}
 \Big| \int (1+|\xi|^2)^q\mathrm e^{2\delta(t)(1+|\xi|^2)^\frac1{2\sigma}}
 \widehat{K(u,u)}(\xi)\bar{\hat u}(\xi)\mathrm d\xi\Big|
  &\leq C_{q,s}\|u\|^3_{ G^\delta_{\sigma,q}}\nonumber\\
  &\leq C_{q,s}\big(\,\|u\|^3_{H^q}+\delta \|
  \mathrm e^{\delta A^\frac1\sigma}\!u\|^3_{H^{q+\frac1{3\sigma}}}\big)\nonumber\\
  &\leq C_{q,s} \big(\,\|u\|^3_{H^q}+\delta \|u\|_{ G^{\delta}_{\sigma,q}}
  \|u\|^2_{ G^{\delta}_{\sigma,q+\frac1{2\sigma}}}\big).
  \end{align}
  The inequality on the last line is due to the Sobolev interpolation inequality
  $$\|f\|_{H^{q+\frac1{3\sigma}}}\leq \|f\|^{\frac13}_{H^{q}}\|f\|^{\frac23}_{H^{q+\frac1{2\sigma}}}.$$
    Plugging estimates (\ref{control1}) and (\ref{control2}) into (\ref{energy-estimate}) yields
  \begin{align}\label{energy-estimate1}
  \frac12\frac{\mathrm d}{\mathrm d t}\|u(t)\|^2_{ G^{\delta(t)}_{\sigma,q}}
  \leq \big(\dot\delta(t)+C \delta(t)\|u(t)\|_{ G^{\delta(t)}_{\sigma,q}}\big)\,
   \|u(t)\|^2_{ G^{\delta(t)}_{\sigma,q+\frac1{2\sigma}}}+ C\|u(t)\|^3_{H^q}.
  \end{align}
  Theorem \ref{global-Sobolev1} guarantees the existence  of global classical solution
   $u\in \mathcal C([0,\infty);H^q)$, whence $\theta(t):=\|u(t)\|_{H^q}\in \mathcal C(\mathbb R^+).$ To
   ensure that the first term on the right-hand side  of (\ref{energy-estimate1}) is negligible, we  can set
  \begin{align}\label{sufficient}
    \dot\delta(t)= -C h(t) \delta(t),
   \end{align}
   or
    \begin{align}
   \delta(t)=\delta_0 \exp \big(-C\int_0^t h(t')\mathrm d t'\big),
   \end{align}
   where $0<\delta_0<1,$ and
   $$h^2(t):=2\|u_0\|_{ G^{\delta_0}_{\sigma,q}}^2+
   2C\int_0^t \theta^3(t')\mathrm d t'.$$
   Thus, we can use the bootstrap argument to ensure that, for any $t\in[0,\infty),$
   \begin{align}\label{globalization}
   \|u(t)\|_{ G^{\delta(t)}_{\sigma,q}}^2\leq h^2(t)=2\|u_0\|_{ G^{\delta_0}_{\sigma,q}}^2+
   2C\int_0^t \theta^3(t')\mathrm d t'.
   \end{align}
   This completes the proof of Theorem \ref{one-component-global}.
   \begin{rema}
 Let $\sigma=1,\,a=2.$ According to the Theorem \ref{local-one-component},  we can obtain the global analyticity in time as well as space. Namely,
  for  real analytic $u_0\in G_1$, there exists a unique  analytic solution to (\ref{B}), i.e.,
     $$u\in C^\omega([0,\infty)\times \mathbb R).$$
     Moreover, for every $t\geq0$, the solution $w(t)$ lies in Gevrey class $G_1$.
   \end{rema}
   \begin{rema}
    When $s>\frac32,\,a=2$, according to Lemma \ref{1-global-11}, we can also have the global analyticity (or Gevrey regularity) solution if the
    initial data $u_0$ is in $G_1$ $($or $G_\sigma$ with $\sigma>1).$ However, we will not prospect the global analyticity or
    Gevrey regularity  when $s=1,\, a=2$ unless we restrict the sign condition on $u_0$ such that $m_0:=u_0-u_{0,xx}$ does
    not change sign.
   \end{rema}
\end{proof}
Similarly, we can obtain the following global Gevrey regularity and analyticity result of the two-component system (\ref{1}), which relies on the global strong solution results---Lemma \ref{2global1}--\ref{2global2}.
\begin{theo}
Let $a=2, \, \kappa\geq 0,\,  s>\frac32,\,\sigma \geq1.$  Assume that
$(u_0,\rho_0)\in G_\sigma(\mathbb R)\times G_{\sigma}(\mathbb R).$
 There exists  a unique
 global solution $(u,\rho)$ of (\ref{1}) in Gevrey class $G_\sigma\times G_\sigma,$ namely, for any $t\geq0,$ $(u(t,\cdot)$ and $\rho(t,\cdot))$ are both in Gevrey class $G_\sigma.$
\end{theo}
\noindent\textbf{Acknowledgements.} This work was
partially supported by NNSFC (No.11271382),  FDCT (No. 098/2013/A3), Guangdong Special Support Program (No. 8-2015), and the key project of NSF of  Guangdong province (No. 2016A030311004).

\phantomsection

\end{document}